\documentclass{amsart}

\usepackage{amsfonts}
\usepackage{amssymb}
\usepackage{amsthm}
\usepackage{amsmath}
\usepackage{graphicx}
\usepackage{lpic}
\usepackage[retainorgcmds]{IEEEtrantools}
\usepackage{multicol}
\usepackage[all]{xy}

\usepackage{geometry}
\usepackage{hyperref}
\newtheorem{thm}{Theorem}[section]

\theoremstyle{definition}
\newtheorem{lem}[thm]{Lemma}
\newtheorem{prop}[thm]{Proposition}
\newtheorem{cor}[thm]{Corollary}

\newtheorem{defn}[thm]{Definition}

\newtheorem*{brmk}{Remark}

\newcommand{\bbH}{\mathbb{H}}

\newcommand{\bbN}{\mathbb{N}}

\newcommand{\bbR}{\mathbb{R}}

\newcommand{\bbZ}{\mathbb{Z}}

\newcommand{\cA}{\mathcal{A}}
\newcommand{\cB}{\mathcal{B}}
\newcommand{\cC}{\mathcal{C}}

\newcommand{\cH}{\mathcal{H}}
\newcommand{\cI}{\mathcal{I}}

\newcommand{\cN}{\mathcal{N}}
\newcommand{\cO}{\mathcal{O}}

\newcommand{\cQ}{\mathcal{Q}}
\newcommand{\cR}{\mathcal{R}}

\newcommand{\cT}{\mathcal{T}}

\newcommand{\gG}{\Gamma}

\newcommand{\gga}{\gamma}
\newcommand{\gd}{\delta}

\newcommand{\bcol}{\begin{multicols}{2}}
\newcommand{\ecol}{\end{multicols}}
\newcommand{\bitem}{\begin{itemize}}
\newcommand{\eitem}{\end{itemize}}
\newcommand{\benum}{\begin{enumerate}}
\newcommand{\eenum}{\end{enumerate}}
\newcommand{\bmath}{\begin{displaymath}}
\newcommand{\emath}{\end{displaymath}}
\newcommand{\bfig}{\begin{figure}}
\newcommand{\efig}{\end{figure}}

\begin{document}

\author[B. Groff]{Bradley W. Groff\\
The University of Illinois at Chicago}
\email{bradleywgroff@gmail.com}
\title[QIs, boundaries and JSJ-trees for rel. hyp. groups]{Quasi-Isometries, Boundaries and JSJ-Decompositions of Relatively Hyperbolic Groups}

\date{\today}

\begin{abstract}
We demonstrate the quasi-isometry invariance of two important geometric structures for relatively hyperbolic groups: the coned space and the cusped space. As applications, we produce a JSJ-decomposition for relatively hyperbolic groups which is invariant under quasi-isometries and outer automorphisms, as well as a related splitting of the quasi-isometry groups of relatively hyperbolic groups.
\end{abstract}

\maketitle
\section{Introduction}

	In \cite{Gro1}, Gromov introduced relatively hyperbolic groups as a generalization
	of $\gd$-hyperbolic groups.  The first author to elaborate on this idea was Farb,
	with a new definition in his thesis \cite{Fa1}. The equivalence between Gromov's
	definition and Farb's \emph{relatively hyperbolic with bounded coset penetration}
	is established in \cite{Bo1}, and the necessity of BCP is demonstrated in \cite{Sz1}.
	Following this, there have been contributions by Bowditch \cite{Bo1} and many others,
	often describing these groups through new definitions \cite{DrSa1,GrMa1,Os1,Si1,Ya1}.
	It is important to note that the notion of relative hyperbolicity only makes sense
	given a chosen collection $\cA$ of subgroups so it is more accurate to say that $\gG$ is hyperbolic relative
	to $\cA$. For instance, hyperbolic groups, which are naturally
	hyperbolic relative to $\{1\}$, can be investigated through their 
	non-trivial relatively hyperbolic structures, e.g. the free group $F(a,b)$ 
	is hyperbolic relative to $\langle [a,b]\rangle$.
	
	These groups have enough structure to make
	them fruitful objects of study and they naturally occur in many contexts throughout
	modern geometric group theory. For instance:
	
	\bitem
		\item $\pi_1(M)$ for $M$ a complete, finite volume manifold 
			with pinched negative sectional curvature is hyperbolic 
			relative to cusp subgroups \cite{Fa2,Bo3};
		\item the fundamental group of a graph of groups with 
			finite edge groups is hyperbolic relative to vertex groups \cite{Bo3};
		\item a limit group is hyperbolic relative to maximal 
			non-cyclic abelian subgroups \cite{Al1,Da1};
		\item a group acting geometrically on CAT(0) spaces with isolated flats 
			is hyperbolic relative to the stabilizers of maximal flats \cite{DrSa1,HrKl1};
		\item a hyperbolic group is hyperbolic  relative to \{1\};
	\eitem
	
	For our purposes, the Groves-Manning description of relative hyperbolicity is 
	ideal because of the combinatorial simplicity of computing path lengths. This
	computation, combined with the geometric insight provided by \cite{Dr1},
	is the crucial component for the proof that quasi-isometries of Cayley graphs
	induce quasi-isometries of cusped spaces.

\medskip

\noindent{\bf Theorem \ref{t:cuspqi}.}  
{\em 	Let $\gG_1$ and $\gG_2$ be finitely generated groups. Suppose that $\gG_1$ 
		is hyperbolic relative to a finite collection $\cA_1$
		such that that no $A\in\cA$ is properly relatively hyperbolic. Let 
		$q:\gG_1\to\gG_2 $ be a quasi-isometry of groups. 
		Then there exists $\cA_2$, a collection of subgroups of $\gG_2$, such
		that the cusped space of $(\gG_1,\cA_1)$ is quasi-isometric to that of $(\gG_2,\cA_2)$.
}
\medskip
	
	We remark that there is an analogue for the coned space (Theorem \ref{t:coneqi}).
	Several corollaries follow from Theorem \ref{t:cuspqi}, including Corollary \ref{c:gensets}: exchanging finite
	generating sets does not change the quasi-isometry type of the cusped space,
	(Definition \ref{d:cspace}). However, our main application of Theorem \ref{t:cuspqi} 
	stems from the following:

\medskip

\noindent{\bf Corollary \ref{c:boundaryhomeo}.}  
		With $(\gG_1,\cA_1)$ and $(\gG_2,\cA_2)$ as in Theorem \ref{t:cuspqi}, 
		the cusped spaces $X(\gG_1,\cA_1)$
		and $X(\gG_2,\cA_2)$ have homeomorphic boundaries.

\medskip
	
	In \cite{Bo2,Bo5}, the actions of hyperbolic and relatively hyperbolic groups on their
	boundaries are analyzed to produce JSJ and peripheral splittings, respectively.
	The topological features of the boundary (cut pairs and cut points, Definition \ref{d:cpcp}) 
	correspond to the limit sets of two-ended and peripheral subgroups in the
	boundary, respectively.
	
	This approach was generalized to include groups acting on continua in \cite{PaSw1}.
	The action is inherited by an $\bbR$-tree which is constructed from the topological
	features of the continuum (again, cut pairs and points). We pause to note that
	even in the context of hyperbolic groups this strategy cannot be extended to 
	produce different splittings from larger
	separating sets such as Cantor sets, \cite{DePa1}.
	An application for this construction to CAT(0) groups appears
	in \cite{PaSw2}. Here, the group action on the CAT(0) boundary is analyzed
	via a variant of the convergence property to determine the isomorphism
	types of the vertex groups. 
	
	In \cite{PaSw1}, the authors suggest that the cut-point/cut-pair tree may be useful
	in the analysis of relatively hyperbolic groups. We demonstrate that this is true.
	On our way to this, we prove that this tree is simplicial in Theorem \ref{t:simptree}. 
	Further, the splitting
	produced satisfies our Definition \ref{d:JSJ} (for more generality, see \cite{GuLe2})
	of a JSJ splitting over elementary subgroups 
	relative to peripheral subgroups, Theorem \ref{t:JSJ}. 
	
	We also note that in Section 1 of \cite{Bo5} a project to extend the results of \cite{Bo2} to
	relatively hyperbolic groups is suggested. Here, the terminology \emph{local cut points}
	is used in place of our \emph{cut pairs}, Definition \ref{d:cpcp}. By using the
	convergence property we are able to identify the vertex groups of the combined
	tree. Since this is a JSJ tree and it is quasi-isometry invariant, we effectively finish
	this project.
\medskip

\noindent{\bf Theorem \ref{t:lastthm}.}  
{\em 	Let $\gG_1$ and $\gG_2$ be finitely generated groups. Suppose additionally 
		that $\gG_1$ is one-ended and hyperbolic relative to the finite collection $\cA_1$ 
		of subgroups such that no $A\in\cA$ is
		properly relatively hyperbolic or contains an infinite torsion subgroup. 
		Let $\cT$ be the cut-point/cut-pair tree
		of $\partial(\gG_1,\cA_1)$. If $f:\gG_1\to\gG_2$ is a quasi-isometry then
		\begin{itemize}
			\item $\cT$ is the cut-point/cut-pair tree for $\gG_2$ with respect to the peripheral structure
			induced by Theorem \ref{t:cuspqi},
			\item if $\text{Stab}_{\gG_1}(v)$ is one of the following types then $\text{Stab}_{\gG_2}(v)$ is of the same type,
			\begin{enumerate}
				\item hyperbolic 2-ended,
				\item peripheral,
				\item relatively QH with finite fiber.
			\end{enumerate}
		\end{itemize}
}
\medskip

Moreover, by \cite[Corollary 7.3]{MaOgYa1} the set of relatively hyperbolic structures on $\gG$ 
forms a partial order and by our condition on the parabolic subgroups, this peripheral
structure is the unique maximal structure. Combining this with Corollary \ref{c:boundaryhomeo}, we see that the action of $\mathrm{Out}(\gG)$ 
induces homeomorphisms of $\partial \gG$ and preserves the JSJ tree. 
Thus, we get the following corollary.

\medskip

\noindent{\bf Corollary \ref{c:outinv}.} 
	The JSJ tree corresponding to the cut-point / cut-pair tree is a fixed point
	for the action of $\mathrm{Out}(\gG)$ on the JSJ deformation space.
\medskip

Lastly, we obtain some results controlling the structure of $\cQ\cI(\gG)$ by observing
that it acts on the same tree. This action (which is shown to be faithful by Theorem \ref{QIfaith}) can be reasonably well understood by inspecting the quasi-isometry classes of the vertex stabilizers. 
Using this, we can find bounds for the number of edges in the quotient graph $\cT/\cQ\cI$ and produce a splitting of $\cQ\cI(\gG)$ which is induced by the splitting of $\gG$, Theorem \ref{t:edges}.

\section{Preliminaries}
The following section is a brief summary of the various definitions and facts
which we will use directly. We also include some discussion relating our results to
other research. Given this brevity, we suggest several more comprehensive resources:
for coarse geometry, $\gd$-hyperbolic spaces and groups, and
their boundaries, see \cite{BrHa1}; for relatively hyperbolic groups, 
\cite{Bo1,Dr1,Fa2,GrMa1,Hr1,Os1}; for constructing group actions on
trees from group actions on continua \cite{Bo2,Bo3,PaSw1}; for JSJ decompositions,
\cite{GuLe2,GuLe3}; and for convergence actions \cite{GeMa1}. For the remainder of this paper, we assume that $\gG$ is generated by the finite set $S$.

\subsection{Relatively Hyperbolic Groups and $X(\gG,\cA)$.}
\begin{defn}\cite{GrMa1}
	Let $T$ be any graph with edges of length 1. We define 
	the \emph{combinatorial horoball} based at $T$, $\cH(=\cH(T))$
	to be the following 1-complex:
	\bitem
		\item $\cH^{(0)}=VT\times\left\{\{0\}\cup\bbN\right\}$
		\item $\cH^{(1)}=\{((t,n),(t,n+1))\}\cup\{((t_1,n),(t_2,n))\mid d_T(t_1,t_2)\le 2^n\}$.
			We call edges of the first set \emph{vertical} and of the second \emph{horizontal}.
	\eitem
\end{defn}
\begin{brmk}
	In \cite{GrMa1}, the combinatorial horoball is described as a 2-complex
	because they needed the complex to be simply connected. As we are not concerned
	with that here, we only require the 1-skeleton.
\end{brmk}
\begin{defn}\cite{GrMa1}
	Let $D:\cH\to [0,\infty)$ be defined by extending the map on vertices $(t,n)\to n$
	linearly across edges. We call $D$ the \emph{depth function} for $\cH$
	and refer to vertices $v$ with $D(v)=n$ as \emph{vertices of depth $n$} or \emph{depth $n$ vertices}.
\end{defn}
\noindent Because $T\times\{0\}$ is homeomorphic to $T$, we identify $T$ with $D^{-1}(0)$.
\begin{defn}\cite{GrMa1}\label{d:cspace}
	Let $\cA$ be a collection of subgroups of $\gG$. The \emph{cusped space} $X(\gG, S, \cA)$ is the union of $\gG$ with
	$\cH(gA)$ for every coset of $A\in\cA$, identifying
	$gA$ with the depth 0 subset of $\cH(gA)$. We suppress mention of $S,\cA$
	when they are clear from the context.
\end{defn}

\noindent For points in $X(\gG)$, we do not distinguish between the depth functions
	of distinct horoballs because horoballs only overlap at depth 0 and so this
	convention is unambiguous. Thus, we can refer
	to the depth of any vertex in $X(\gG)$ without mention of the
	associated horoball or coset.

\begin{defn}
	The elements of the collection of subgroups $\cA$ are called \emph{parabolic
	subgroups} and the subgroups which are conjugate to them are called
	\emph{peripheral subgroups}.
\end{defn}

\begin{defn}\cite{GrMa1}
	If $X(\gG, S,\cA)$
	is $\gd$-hyperbolic then we say that \emph{$\gG$ is hyperbolic relative to $\cA$} or
	that \emph{$\gG$ is a relatively hyperbolic group} or that \emph{the pair $(\gG,\cA)$
	is relatively hyperbolic}.
\end{defn}
\begin{brmk}
	As mentioned in the introduction, there are many definitions of relatively hyperbolic groups.
	These definitions are all equivalent, with the exception of Farb's in which the phrase
	``with bounded coset penetration" needs to be appended. Equivalence of Definition \ref{d:cspace}
	with others is established in \cite[Theorem 3.25]{GrMa1}.
\end{brmk}
\begin{brmk}
	Substituting $S$ for some other finite generating set $S'$ changes the topology
	of $X(\gG, S, \cA)$ and may change the value of $\gd$, but does not affect
	the hyperbolicity of the cusped space for some $\gd'$.
	Thus we omit $S$ from the definition above. See \cite{GrMa1}, or \cite{Os1}
	for this result in a different context.
\end{brmk}
\begin{brmk}
	In \cite[Section 4]{Hr1}, the construction of the cusped space is extended to groups
	which only have a finite \emph{relative generating set} (ie $\gG=\langle S\cup H\rangle$)
	by declaring a proper metric on peripheral subgroups. It seems reasonable to generalize
	our results in this fashion.
\end{brmk}
\begin{defn}
	\cite{Bo1} Given a group $\gG$ hyperbolic relative to $\cA$, the \emph{Bowditch boundary},
	$\partial(\gG, S, \cA)$ is the
	Gromov boundary of the associated cusp space, $\partial X(\gG, S,\cA)$. When
	there is no ambiguity we simply say the boundary.
\end{defn}
\begin{brmk}
	In \cite{Bo1}, the boundary is defined as the ideal boundary of a proper, hyperbolic
	space on which the group acts. This is largely motivated by the definition 
	for relatively hyperbolic groups given in \cite{Gro1}.
	Part of the appeal of the characterization of relative hyperbolicity given in \cite{GrMa1}
	is that it satisfies the conditions of the definition of \cite{Gro1}, and so it naturally fits
	with the notion of boundary developed in \cite{Bo1}. In fact, Bowditch constructs
	a very similar space to the cusped space in which the edges at depth $n$ are shrunk by a factor of
	$2^n$ instead of adding extra edges. The two spaces are quasi-isometric by the natural
	map on vertices, so the boundaries are the same.
\end{brmk}
\begin{lem}\cite[3.11]{GrMa1}\label{l:horoballboundary}
	If $A$ is a combinatorial horoball, then the Gromov boundary consists of a single point,
	denoted $e_A$, which can be represented by a geodesic ray containing only vertical
	edges.
\end{lem}
Because our proof of Theorem \ref{t:cuspqi} can be simplified to prove an analogous result
for the coned space, we include its definition.
\begin{defn}\cite{Bo1}
	Let $G$ be a group with a finite collection of subgroups $\cB$. For every coset $gB_i$
	of an element of $\cB$, we add to the Cayley graph a vertex $v_{gB_i}$ and from every element of $gB_i$
	we add an edge to $v_{gB_i}$. We call this the \emph{coned space} of $(G,\cB)$.
\end{defn}
\begin{defn}\cite{Bo1}
	A graph is \emph{fine} if for every $n>0$ and every edge $e$, the number of cycles 
	of length at most $n$ containing $e$ is finite.
\end{defn}

\begin{defn}[\cite{Bo1} Alternate characterization of relative hyperbolicity]
	Let $C$ be the Cayley graph of a group $G$ generated by a finite set $S$ and 
	let $\cA$ be a collection of subgroups of $G$. If the coned space of $(G,\cA)$
	is hyperbolic and fine, then $(G,\cA)$ is relatively hyperbolic \cite[p. 27]{Bo1}.
\end{defn}
The following three definitions appear in Chapter 4 of \cite{Os1}.
\begin{defn} An element $g$ of
	is called \emph{hyperbolic} if $g$ is not conjugate into any $A_i\in\cA$.
\end{defn}
\begin{defn}
	A subgroup $H$ of a relatively hyperbolic group $\gG=\langle S|R\rangle$ is called 
	\emph{relatively quasi-convex} if there exists $\sigma>0$ such
	that the following condition holds. Let $f,g$ be two elements of $H$, 
	and $p$ an arbitrary geodesic path from $f$ to $g$ in $\mathrm{Cay}(\gG,S\cup\cA)$.
	Then for any vertex $v \in p$, there exists a vertex $w \in H$ such that
	$dist_{S}(u,w)\leq \sigma$.
\end{defn}
\begin{defn}
	If $H$ is relatively quasi-convex and $H\cap A^g$ is finite for all $g\in\gG$ and $A\in\cA$, then
	$H$ is called \emph{strongly relatively quasi-convex}.
\end{defn}

\begin{thm}[\cite{Os1}, Theorem 4.19]\label{t:2endsrq}
	The centralizer of a hyperbolic element in a relatively hyperbolic group is strongly relatively quasi-convex.
\end{thm}

\begin{defn}
	A subgroup of a relatively hyperbolic group is called \emph{elementary} if
	it is either virtually cyclic or a subgroup of a peripheral subgroup.
\end{defn}


\subsection{JSJ Decompositions}

JSJ decompositions were first introduced in the context of 3-manifold topology
in the works of Jaco and Shalen and, independently, Johannson \cite{JaSh1,Jo1}. 
The manifold is separated by a maximal system of disjoint, embedded tori 
with the goal of understanding the structure of the ambient manifold 
by inspecting the individual components.

This idea was ported to geometric group theory originally by Kropholler
\cite{Kr1} and then by Rips and Sela, \cite{RiSe1}. Several authors have expanded on these
ideas, including \cite{Bo2,DuSa1,FuPa1,Pa1,PaSw2}, and, 
in a slightly different manner, \cite{ScSw1}. In the group theoretic setting, 
JSJ decompositions encode a maximal amount of the
information related to two-ended splittings (or, in the case of \cite{RiSe1}, $\bbZ$-splittings)
as a graph of groups. 

Several of these notions have recently been unified 
by the results in \cite{GuLe1,GuLe2,GuLe3}. It is the language presented
here which we adopt. Since we are interested only in a particular type of JSJ-decomposition,
we do not include the most general definitions.

\begin{defn}\label{d:JSJ}
	Let $\gG$ be hyperbolic relative to $\cA$. An \emph{elementary JSJ splitting relative
	to $\cA$} is a tree, $T$, with a $\gG$ action such that the following hold.
	\begin{enumerate}
		\item all edge stabilizers are elementary subgroups;
		\item (universally elliptic) any edge stabilizer of $T$ fixes a 
			point in any other tree with property (1);
		\item (maximal for domination) for any tree $T'$ satisfying (2), 
			every vertex stabilizer of $T$ stabilizes a vertex of $T'$; and
		\item (relative to $\cA$) all subgroups of elements of $\cA$ fix a point in T.
	\end{enumerate}
\end{defn}

In \cite{GuLe2}, the authors promote the notion that the correct objects of study for JSJ
decompositions should not be JSJ trees because they are not unique. Rather,
a collection of trees is more fundamental:
\begin{defn}
	A maximal collection of trees which all satisfy Definition \ref{d:JSJ} is called a
	\emph{JSJ deformation space}.
\end{defn}

In the study of JSJ decompositions, one important focus is on understanding
those subgroups in which there are many mutually incompatible 
splittings. Such pairs of splittings are called \emph{hyperbolic-hyperbolic} in
\cite{RiSe1} and are best understood as analogous to splittings of a surface group over two simple
closed curves with an essential intersection. It is impossible to realize both splittings
simultaneously with a common refinement of the graphs of groups.

The subgroups with this property have various names in different contexts, including
\emph{quadratically hanging} \cite{RiSe1}, \emph{maximal hanging Fuchsian}
\cite{Bo2}, \emph{orbifold hanging vertex} \cite{Pa1}. These names all seek to
describe the same central idea: many pairs of simple closed curves on surfaces
intersect. Consequently, surface groups have many $\bbZ$-splittings which can not be
simultaneously realized in a common graph of groups. The essential power of these ideas
is that whenever these `hyperbolic-hyperbolic' splittings occur in finitely presented groups, the 
situation is always very close to the surface case. There is a more general definition
given in \cite{GuLe2} which happens to encompass all of the above.

\begin{defn}\label{d:flexible}
	Let $\gG_v$ be a vertex group for a JSJ tree. If $\gG_v$ is not universally
	elliptic, then $\gG_v$ is called \emph{flexible}.
\end{defn}

For our purposes, this definition falls 
somewhat short. We are interested
in identifying subgroups up to quasi-isometry and flexibility is not preserved
under such maps. For example, compare closed hyperbolic surface 
groups with hyperbolic triangle groups. In the context of relatively hyperbolic
groups, there is a larger class of subgroups which we can use instead of flexible
subgroups and which will reflect the coarse geometry directly: relatively QH
subgroups with finite fiber \cite{GuLe2}. 

\begin{defn}\label{d:QH}
	Given a group with a JSJ tree relative to $\cA$, a subgroup $Q$ is a 
	\emph{relatively QH-subgroup} if it satisfies the following:
	\begin{enumerate}
		\item an exact sequence, with $\cO$ a hyperbolic 2-orbifold, $F$ called the \emph{fiber}:
			$$1\to F\to Q \to \pi_1(\cO)\to1$$
		\item the images of incident edge groups are either finite or contained
			in a boundary subgroup of $\pi_1(\cO)$.
		\item every conjugate of an element of $\cA$ intersects $Q$ with image
			either finite or contained in a boundary component of $\pi_1(\cO)$.
	\end{enumerate}
\end{defn}

Specifically, these subgroups will be detectable from the boundary of the group,
and thus will be realized as vertex groups for the cut-point/cut-pair tree.

\subsection{Convergence Actions}

Suppose $\gG$ acts on a space $X$ by homeomorphisms. 

\begin{defn} A sequence $(g_i)$ is a \emph{convergence sequence on X} if there
exists points $x_1,x_2\in X$ such that for any compact $C$ not containing $x_1$,
$g_i(C)\to x_2$.
\end{defn}

\begin{defn} If every sequence in $G$ contains a convergence subsequence, then $G$ acts
as a convergence group on $X$.
\end{defn}

Both hyperbolic and relatively hyperbolic groups are characterized
by types of convergence actions which they exhibit on their boundaries \cite{Bo3,Ya1}.
This allows for the identification of stabilizers of certain topological features in these
contexts. We also require the classic theorem

\begin{thm}[\cite{CaJu1, Ga1, Tu1}]\label{t:convergence}
	Let $G$ be a subgroup of $\mathrm{Homeo}(S^1)$. $G$ acts as a convergence group
	on $S^1$ if and only if $G$ is Fuchsian.
\end{thm}

\noindent A \emph{Fuchsian group} is a discrete subgroup of 
M\" obius transformations of $D^2$. In the proof of Theorem 
\ref{t:cuspqi} we will use the convergence action to identify
the vertex groups of our splitting.

\section{The Tree from the Boundary} 

Given a group acting on a continuum by homeomorphisms, \cite{PaSw1} constructs an $\bbR$-tree with vertices
representing the topological features of the continuum (cut points and cut pairs).
The tree inherits the action of the group on the continuum. We condense their
exposition in anticipation of demonstrating that this tree is simplicial and of JSJ type in the context of relative
hyperbolicity in Sections \ref{s:Rips} and \ref{s:JSJ}. For the remainder of this section we assume that $\gG$ is a relatively
hyperbolic group.

\begin{defn}
	A \emph{continuum} is a compact connected metric space.
\end{defn}

\begin{defn} \label{d:cpcp}Given a continuum $X$, a point $x\in X$ is a \emph{cut point} if
	$X\setminus \{x\}$ is not connected. If $\{a,b\}\subset X$ contains no cut points and
	$X\setminus\{a,b\}$ is not connected, then $\{a,b\}$ is a \emph{cut pair}. A set $Y$ is
	called \emph{inseparable} if no two points of $Y$ lie in different components of the 
	complement of any cut pair.
\end{defn}

These separating features occur as the fixed points of peripheral or hyperbolic two-ended
subgroups over which $\gG$ splits. Given that we want to also understand when there are many mutually
incompatible splittings (as in Definitions \ref{d:flexible} and \ref{d:QH}), we have terminology reflecting
interlocking cut pairs. These pairs arise in our context as the
endpoints of pairs of hyperbolic two-ended subgroups over which the group
splits but which admit no common refinement.

\begin{defn}
	Let $X$ be a continuum without cut points. A finite set $S$ is called
	a \emph{cyclic subset} if there is an ordering $S=\{s_1,s_2,\ldots,s_n\}$
	and continua $M_1,\ldots M_n$ such that 
	\begin{enumerate}
		\item $M_i\cap M_{i+1} = \{s_i\}$, subscripts mod $n$
		\item $M_i\cap M_j = \emptyset$ whenever $|i-j|>1$
		\item $\cup M_i = X$
	\end{enumerate}
	An infinite subset in which all finite subsets of cardinality at least 2 are
	cyclic is also called cyclic. 
\end{defn}

\begin{defn}
	A maximal cyclic subset with at least 3 elements is called a \emph{necklace}.
\end{defn}

Cyclic subsets arise as collections of mutually separable cut pairs. We also note
that an inseparable cut pair
can be in the closure of more than one necklace, but if the cut pair is not inseparable
then that necklace is unique.

Given a continuum $X$, we define an equivalence relation $\sim$ such that
any cut point is equivalent only to itself and for $x,y$ which are not cut pairs, $x\sim y$
if and only if there is no cut point $z$ such that $x$ and $y$ are in different components
of $X\setminus z$.

We would like to define a similar notion for cut pairs but the extra structure
(particularly the interaction between maximal inseparable sets and necklaces) 
makes this difficult. Instead, we directly construct subsets of the powerset of $X$ which
reflect the topology. Let $\cR$ be the collection containing all inseparable cut pairs,
necklaces and maximal inseparable sets of $X$. We claim that this structure
is compatible with $\sim$, ie that $\cR$ is the union of sets defined similarly
on each class of $\sim$. This follows from the following lemma, that cut points do not 
separate cut pairs.

\begin{lem}\label{l:cpsepcp}
	Suppose that $T$ is a connected topological space with cut point $x$. If $\{y,z\}$
	is a cut pair then $y$ and $z$ are in the same component of $T\setminus x$.
\end{lem}
\begin{proof}
	Let $C_1$ be the component of $T\setminus \{y,z\}$ containing $x$. 
	Let $w$ be a point in another component, $C_2$. Clearly
	$x$ separates $C_1$ but not $C_2$. Thus, $w$ and $y$ are in 
	the same component of $T\setminus\{x,z\}$ and $w$ and $z$ are
	in the same component of $T\setminus\{x,y\}$. Thus, $y$ and $z$
	are in the same component of $T\setminus x$.
\end{proof}

In \cite[Theorems 12, 13, 14]{PaSw1}, $\sim$ is shown to satisfy a `betweenness'
property so that a process of `connecting the dots' can fill it in to an $\bbR$-tree.
\cite[Corollary 31]{PaSw1} serves the same purpose for $\cR$. The combination of
the two (which Lemma \ref{l:cpsepcp} justifies)
is discussed in Section 5 of \cite{PaSw1}, see Figure \ref{f:continuum}.
Obviously a group action by homeomorphisms on a continuum is
inherited by this tree, although the action is not \emph{a priori} isometric.
In fact, this $\bbR$-tree does not necessarily come equipped with a metric.

	\begin{figure}[t]
	\centering
	\begin{lpic}[]{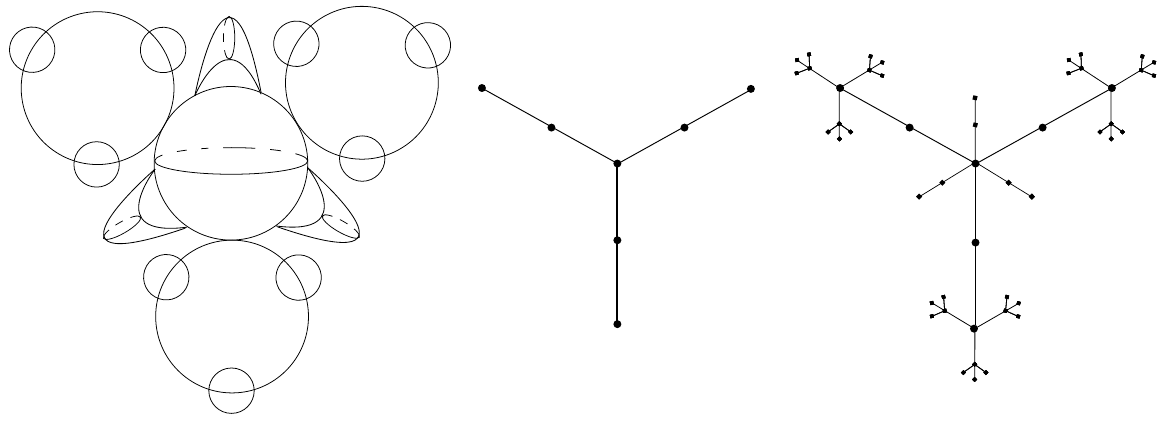}
		
	\end{lpic}
	\caption{A continuum with the associated cut-point tree and combined tree.\label{f:continuum}}
	\end{figure}

Of course, the most important question about an action of a group on
an $\bbR$-tree is whether it can be promoted to an action on a simplicial
tree, or whether the $\bbR$-tree itself is simplicial. This is often accomplished 
via the Rips machine, which can be applied
in the context of a reasonably nice action.

\begin{defn}
	Let $\gG$ act on the $\bbR$-tree $T$ by homeomorphisms. We say the action is \emph{nesting}
	if there exists a $g\in\gG$ and an interval $I\subset T$ such that $g(I)$ is
	properly contained in $I$. Otherwise, we say the action is \emph{non-nesting}.
	
	A non-degenerate arc $I$ is called \emph{stable} if there is a non-degenerate sub-arc 
	such that for any non-degenerate arc $K\subset J$, 
	$\mathrm{Stab}(K) = \mathrm{Stab}(J)$. An action is called \emph{stable} 
	if any closed arc I of T is stable.
\end{defn}

	We remark that the cut-point tree (ie the tree produced by connecting the dots
	for $\sim$) is simplicial whenever the boundary is
	connected and locally connected \cite[Theorem 9.2]{Bo5}. This can be achieved by the following
	mild constraints on $\cA$.
	
	\begin{thm}[\cite{Bo5}, Theorem 1.5]
		Suppose that $\gG$ is relatively hyperbolic and that each peripheral subgroup is
		one- or two-ended and contains no infinite torsion subgroup. If $\gG$ is
		connected then it is locally connected.
	\end{thm}
	
	We can produce a tree by `blowing up' the vertices of subcontinua with no cut
	points according to the cut-pair structure ($\cR$), as discussed in \cite[Section 5]{PaSw1}.
	We call this the \emph{combined tree} or \emph{cut-point / cut-pair tree}, which we will
	denote by $\cT$ (Figure \ref{f:continuum}). As mentioned above,
	we seek to apply the Rips machine so we must demonstrate that the action
	is stable and non-nesting.

\section{$\mathcal{T}(\partial X(\gG))$ is Simplicial}\label{s:Rips}

	Because the cut point tree
	is simplicial, the only intervals which can be unstable are those which contain
	multiple inseparable cut pairs. Thus, we restrict our attention to those. We 
	first require some information on the
variety of finite subgroups of a relatively hyperbolic group.

\begin{lem}\label{l:conjclass}
	Let $\gG$ be a relatively hyperbolic group. There are finitely many conjugacy classes
	of finite order subgroups $F$ such that $F$ is contained in a hyperbolic two-ended 
	subgroup $H$ with $F$ fixing the ends of $H$.
\end{lem}
\begin{proof}
	The following argument is adapted from \cite{MoMO}. For every hyperbolic
	two-ended subgroup $H$, take $A_H\subset X(\gG,\cA)$ to be the set of all geodesics between
	the endpoints of $H$. There is a uniform width $W$ for all $A_H$ which is 
	independent of the choice of $H$ \cite[Corollary 8.16]{Hr1}, ie for a point 
	$p\in A_H\cap \mathrm{Cay}(\gG)\subset X(\gG,\cA)$, 
	$H\setminus N_W(p)$ is not connected.
	
	It follows that for some $k\in\bbN$, if $h\in H$ has the property that 
	$[h.N_{kW}(p)]\cap N_{kW}(p) = \emptyset$ and $h$ fixes the endpoints of
	$H$, then $h$ has infinite order.  
	
	For $F$ as above, it must be that the $F$-translates of $N_{kW}(p)$ are not
	disjoint from $N_{kW}(p)$. Conjugating by $p$ sends every such $F$ to
	a subgroup $F'\subset N_{kW}(1)$ of the same conjugacy class, and there
	are only finitely many possible $F'$.
	\end{proof}

\begin{lem}\label{l:uniformbound}
There is a uniform bound on the order of the stabilizer of an interval containing at
least two vertices corresponding to cut pairs.
\end{lem}
\begin{proof}

	Let $I$ be any interval containing at least two 
	inseparable cut pairs, $A, B$. We show that $\mathrm{Stab}(I)$ is finite. If 
	$\{g_n\}\subset\mathrm{Stab}(I)$ is an infinite sequence of elements then
	we may assume that $g_n(x)\to p$ for all $x\in\partial\gG$, perhaps after 
	passing to a subsequence. However, each $g_n$ must fix all of the points
	of both cut pairs. This is a contradiction.
	
	Additionally, this finite subgroup satisfies the hypotheses of Lemma \ref{l:conjclass}
	because it is a subgroup of $\mathrm{Stab}(A)$. There are only finitely
	many conjugacy classes of such subgroups so there is a uniform bound
	on the order of $\mathrm{Stab}(I)$. It follows immediately that the action is stable.
\end{proof}

\begin{cor}\label{c:stable}
	The action of $\gG$ on $\cT$ is stable.
\end{cor}
\begin{lem}\label{l:nonnesting}
	The action of $\gG$ on $\cT$ is non-nesting.
\end{lem}
\begin{proof}
	Assume not. Then there exists an interval $I\subset T$ and $g\in\gG$ (replaced by
	$g^2$ if necessary) such
	that $g(I)$ is a proper subset of $I$. By the Brouwer fixed point theorem, there
	is a fixed point of $g$ in $I$, call this $A$. We may assume $I=[A,B]$,
	$g(B)\in[A,B)$. Note that $g$ has infinite order.
	
	By convergence, there exists $p,q\in\partial \gG$ such that $g^n(x)\to p$ for
	every $x\neq q$. Clearly, $p\in A$. Because $g^{-n}(x)\to q$, also $q\in A$ otherwise
	$q\in C\in (A,B)$ but $g^{-n}(C)\neq C$.
	However, this implies that for all $x\neq p$, $g^{-n}(x)\to q\in A$. Yet
	$g^{-n}(B)\not\in[A,B]$ for any $n$, a contradiction.
\end{proof}

\begin{lem}\label{l:minimal}
	This action is also minimal.
\end{lem}
\begin{proof}
	Every cut point is stabilized by a peripheral subgroup and
	every cut pair is stabilized by a finite or two-ended subgroup.
	Since the group is one-ended, cut pair stabilizers must be two-ended, showing
	that a 2-dense collection of vertices of the cut point tree and of 
	every cut pair tree have non-trivial
	stabilizers. Thus, no proper invariant subtree exists.
\end{proof}

\begin{thm}\label{t:simptree}
	Let $\gG$ be a finitely presented, one-ended group, hyperbolic relative to $\cA$ such that for
	every $A\in\cA$, $A$ is not properly relatively hyperbolic and $A$ contains no
	infinite torsion subgroup. Let $\cT$ be the combined 
	tree obtained by the action of $\gG$ on its Bowditch boundary.
	Then $\cT$ is simplicial.
\end{thm}
\begin{proof}
	
	Suppose not. Then because the action is non-nesting, by \cite{Le1}, 
	there is an $\bbR$-tree $\cT'$ equipped with an isometric $\gG$-action 
	and an equivariant quotient map $\cT\to \cT'$. Furthermore, stabilizers of 
	non-simplicial segments
	in $\cT'$ stabilize segments in $\cT$, and so are finite of uniformly bounded order 
	by Lemma \ref{l:uniformbound}.
	Therefore, as in Corollary \ref{c:stable}, the $\gG$-action is stable.

In all cases of \cite[Theorem 9.5]{BeFe1} other than the pure surface case,
one obtains a splitting over a finite group.  However, $\gG$ is
one-ended, so we reduce to this case.  By \cite[Theorems 9.4(1) \& 9.5]{BeFe1}
$\gG$ admits a splitting over a two-ended group $V$, and
this two-ended group corresponds to an essential, non-boundary
parallel simple closed curve in the associated orbifold.  If $g \in V$
corresponds to this curve, then since the associated lamination on the
orbifold has no closed leaves $g$ must act hyperbolically on $\cT'$.
This implies that $g$ also acts hyperbolically on $\cT$.  However, a
splitting of $\gG$ over a two-ended group must induce a cut pair
corresponding to the endpoints of the axis of $\langle g \rangle$.
This cut pair must be stabilized by $g$, so $g$ cannot act
hyperbolically.  This is a contradiction.
\end{proof}

In summary, the combined tree $\cT$ is simplicial and has one vertex for each of the following topological
structures in the continuum:
\begin{enumerate}
	\item cut points
	\item inseparable cut pairs
	\item necklaces
	\item equivalence classes of points not separated by cut points or cut pairs
\end{enumerate}

\noindent Additionally, there is an edge between two vertices if the corresponding sets in the continua
have intersecting closures. We note that, by the construction of $\cR$, points of the continuum
can be contained in multiple elements of $\cR$.

\section{$\cT(\gG,\cA)$ is a JSJ-Tree}\label{s:JSJ}

Now that we know that $\cT$ is simplicial in common situations, we are left
with the task of classifying it as a JSJ tree. This effectively labels the splitting
as the ideal splitting for understanding the algebraic content of the group.

\begin{thm}\label{t:JSJ}
	With $\gG,\cA$ as in Theorem \ref{t:simptree}, $\cT$
	is a JSJ tree over elementary subgroups relative to peripheral subgroups.
\end{thm}

\begin{proof} We show that $\cT$ satisfies the conditions of Definition \ref{d:JSJ}. 
By the construction of the tree every edge group must be the stabilizer
of either a cut point or a cut pair. Because relatively hyperbolic groups act on their
boundaries with a convergence action, these stabilizers must be elementary
subgroups (condition (1)). Every peripheral subgroup fixes a point in the tree because it fixes
a point in the boundary (this point is $e_A$ of Lemma \ref{l:horoballboundary}), which implies 
condition (4). Furthermore, this tree satisfies (3) because every
elementary splitting always has a topological expression in the boundary. In particular,
\cite{Bo5} implies the existence of a cut point and \cite{Pa1} implies the existence
of a cut pair whenever there is an peripheral or hyperbolic two-ended 
splitting, respectively. Thus, every vertex
in every such tree comes from one of these structures and hence is already a vertex
stabilizer in $T$. Finally, every splitting of this group must reflect the 
topology of the boundary. In particular, every edge represents the intersection
of topological features. Fixing these features means fixing a vertex in a tree
for another splitting. In essence, the argument that this tree satisfies (3) goes
in both directions. This is (2).
\end{proof}

\section{Induced Quasi-Isometries of Cusped Spaces}

	In this section we prove that quasi-isometries of Cayley graphs induce quasi-isometries
	of cusped spaces:

\medskip

\noindent{\bf Theorem \ref{t:cuspqi}.}  
{\em 
Let $\gG_1$ and $\gG_2$ be finitely generated groups. Suppose that $\gG_1$ 
		is hyperbolic relative to a finite collection $\cA_1$
		such that that no $A\in\cA$ is properly relatively hyperbolic. Let 
		$q:\gG_1\to\gG_2 $ be a quasi-isometry of groups. 
		Then there exists $\cA_2$, a collection of subgroups of $\gG_2$, such
		that the cusped space of $(\gG_1,\cA_1)$ is quasi-isometric to that of $(\gG_2,\cA_2)$.

}
\medskip

We first show that horoballs of quasi-isometric spaces are themselves quasi-isometric.
To that end, we distinguish among types of geodesics which exist in horoballs. We assume
that $n_2>n_1$.

\begin{defn}
	Let $\hat{T}$ be a horoball over the graph $T$ with $(t_1,n_1)$ and
	$(t_2,n_2)$ vertices of $\hat{T}$. We say $[(t_1,n_1),(t_2,n_2)]$ is
	\emph{vertical} or a \emph{vertical geodesic} 
	if $n_2$ is the maximal depth among vertices of $[(t_1,n_1),(t_2,n_2)]$. See Figure \ref{f:vertgeod}.
\end{defn}

	\begin{figure}[h]
	\centering
	\includegraphics[width=0.9\textwidth]{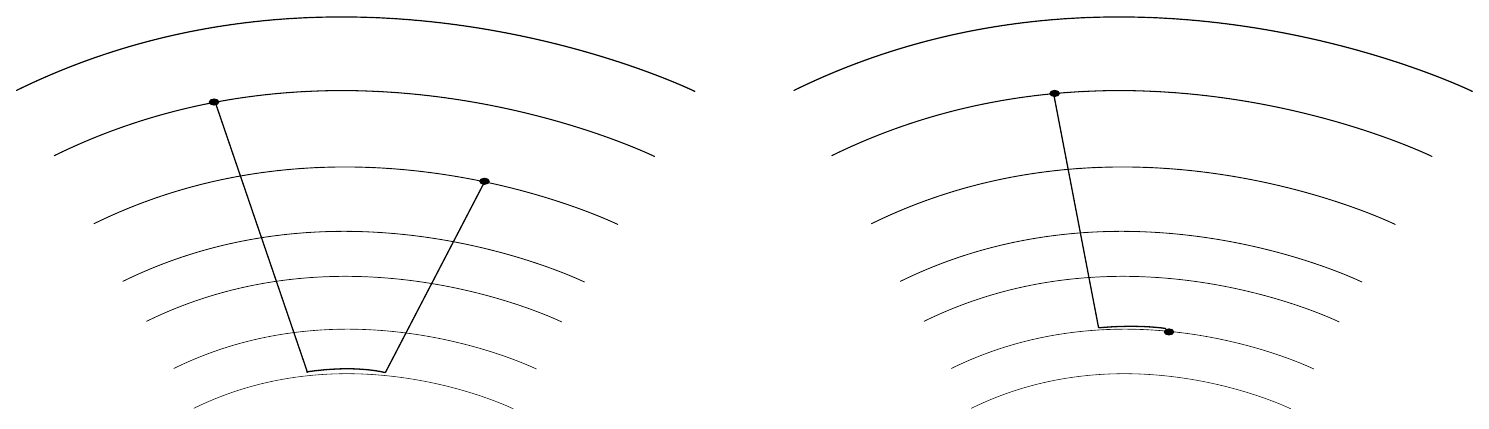}

	\caption{A vertical geodesic (right) and a non-vertical geodesic.\label{f:vertgeod}}	\end{figure}

\begin{lem}\label{l:horoballqi}Let $q:T\to S$ a $(k,c)$-quasi-isometry between graphs. 
	There is a $(1,C)$-quasi-isometry
	$\hat{q}:H(T)\to H(S)$ between combinatorial horoballs
	such that $\hat{q}$ extends $q$. Furthermore, $C$ 
	depends only on $k$ and $c$.
\end{lem}

\begin{proof}
	Extend $q$ to $\hat{q}$ by defining $\hat{q}(v,n)=(q(v),n)$ and let $s_i=q(t_i)$.
	The proof proceeds by comparing lengths of geodesics; we show that the length of a
	segment in $H(S)$ is
	less than a linear function of the length of the corresponding segment in $H(T)$.
	Since we start with a quasi-isometry, this argument is also valid in the reverse direction.
	Moreover, the quasi-inverse has coefficients which obey the same dependencies which establishes the proof. We partition the proof into cases by which
	geodesics are vertical.
	
	\emph{$[(s_1,n_1),(s_2,n_2)]$ vertical:} $d_{H(S)}\in\{n_2-n_1,n_2-n_1+1,n_2-n_1+2,n_2-n_1+3\}$ 
	and $d_{H(T)}+3$ is clearly at least as large.
	
	For the remaining two cases we assume that $[(s_1,n_1),(s_2,n_2)]$ is not vertical.
	
	\emph{$[(t_1,n_1),(t_2,n_2)]$ not vertical:}
\begin{align}
	d_{H(S)}((s_1,n_1),(s_2,n_2)) &\leq2\log_2\left[d_S(s_1,s_2)\right]+3-n_2-n_1\nonumber\\
	&\leq2\log_2[kd_T(t_1,t_2)+c]+3-n_2-n_1\nonumber\\
	&\leq2\log_2[(k+c)d_T(t_1,t_2)]+3-n_2-n_1\nonumber\\
	&\leq2\log_2(k+c)+2\log_2[d_T(t_1,t_2)]+3-n_2-n_1\nonumber\\
	&\leq\hat{d}_T((t_1,n_1),(t_2,n_2))+2\log_2(k+c)+3\nonumber
\end{align}

	\emph{$[(t_1,n_1),(t_2,n_2)]$ vertical:}
\begin{align}
	d_{H(S)}((s_1,n_1),(s_2,n_2))&\leq2\log_2(d_S(s_1,s_2))-n_2-n_1+3\nonumber\\
	&\leq 2\log_2(k+c)+2\log_2(d_T(t_1,t_2))-n_2-n_1+3\nonumber\\
	&\leq2\log_2(k+c)+n_2-n_1+3 \\
	&\leq2\log_2(k+c)+\hat{d}_T(t_1,t_2)+3 \nonumber
\end{align}	
	Since $[t_1,t_2]$ is not descending, $\log_2(d_T(t_1,t_2))\leq n_2$, justifying (2.1).
	
	The coarse density of the image is clear. Since $q$ is a 
	quasi-isometry, we can also get identical results in the reverse direction
	with a symmetric argument so that geodesics in $H(T)$ are bounded
	by a linear function of the lengths in $H(S)$.
	Thus, $\hat{q}$ is a $(1,2\log_2(k+c)+3)$-quasi-isometry. 
		\end{proof}

	\begin{thm} \label{t:cuspqi}
		Let $\gG_1$ and $\gG_2$ be finitely generated groups. Suppose that $\gG_1$ 
		is hyperbolic relative to a finite collection $\cA_1$
		such that that no $A\in\cA$ is properly relatively hyperbolic. Let 
		$q:\gG_1\to\gG_2 $ be a quasi-isometry of groups. 
		Then there exists $\cA_2$, a collection of subgroups of $\gG_2$, such
		that the cusped space of $(\gG_1,\cA_1)$ is quasi-isometric to that of $(\gG_2,\cA_2)$.

	\end{thm}
\begin{proof}
	We extend $q$ to a map $Q:X(\gG_1)\to X(\gG_2)$. First, let $Q=q$ on $\mathrm{Cay}(\gG_1)$. 
	By the proof of Theorem 5.12 of \cite{Dr1}, $q$ induces a quasi-isometric
	embedding of cosets of elements of $\cA_1$ into those of $\cA_2$ and
	we can take these to have uniform constants. We observe that this can
	be made coarsely surjective because no peripheral subgroup is properly
	relatively hyperbolic. The reason for this is that if we had 
	a quasi-isometry which was not coarsely surjective then this would
	give an element of $\cA_2$ that would be hyperbolic relative to the
	image of $q$ of some coset of an element of $\cA_1$, again by \cite{Dr1}, 
	contrary to our hypothesis. When the map is coarsely surjective, we get
	that this element is hyperbolic relative to itself, which is the trivial
	case that we allow.
	
	Now, $q$ might
	only take a coset to within a bounded distance of the corresponding coset in $\gG_2$, rather
	than directly to it as in Lemma \ref{l:horoballqi}.
	We can still use the induced quasi-isometry on the subsets of the horoballs which
	have positive depth but at depth 0 we have to make an adjustment. 
	
	However, the proof of \cite[Theorem 5.12]{Dr1} shows us that there exists a bound $T$ such that 
	the image of $A_i$ is at most $T$ from the coset to which it is quasi-isometric.
	Thus, we have to account only for 
	an extra additive $T$ in the constants. 
	
	Thus we know that $q$ induces quasi-isometries 
	between cosets of peripheral subgroups and that the constants of these quasi-isometries
	do not depend on the particular cosets but only on the constants of 
	$q$ and on $(\gG_1,\cA)$. Therefore, $Q$ restricts to a quasi-isometry on each individual horoball and on the Cayley graph. We only need to show that these can be combined
	across all of $X$.
	
	Now let $[x,y]$ be a geodesic arc between points of $X(\gG_1)$. We divide
	this arc into several subarcs by taking the collection of maximal subarcs $I_1$ which 
	have every vertex of depth 0 and also take the complimentary collection of 
	segments $I_2$.  We note that some of these arcs may be degenerate (length 0, just a
	singleton) and that our methods account for this. 
	In other words, we have 
	$$[x,y]=[x_0,x_1]\cup[x_1,x_2]\cup\ldots\cup[x_{n-1},x_n]=[x_0,x_n]$$
	with $[x_{2i},x_{2i+1}]\in I_1$ and $[x_{2i+1},x_{2i+2}]\in I_2$. Essentially, we have
	divided $[x_0,x_n]$ into segments which go between 
	two different horoballs (again, possibly with length 0) and 
	segments which traverse individual horoballs. We should mention
	that the subdivision used here has a parity which suggests that
	the terminal points $x_0$ and $x_n$ must have depth 0, but that this is easily
	surmounted. Simply attach a vertical segment to a depth 0 vertex whenever necessary.
		
	We have the following expression:
	
	$$\hat{d}_1(x_0,x_n)=\sum_{i=0}^{n-1}\hat{d}(x_i,x_{i+1})=\sum_{I\in I_1}\mathrm{length}(I)+\sum_{I\in I_2}\mathrm{length}(I)$$
	
	\begin{figure}[b]
	\centering
	\includegraphics[width=0.8\textwidth]{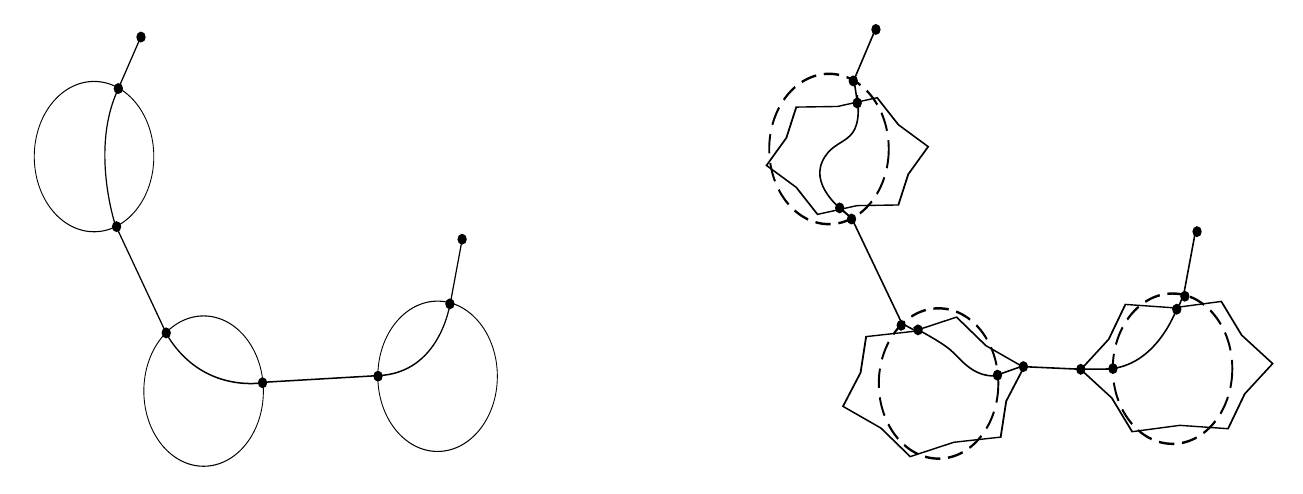}
	\caption{A typical geodesic in $X(\gG_1)$ and a reconstructed piecewise geodesic in $X(\gG_2)$.}\label{f:recongeod}
	\end{figure}
	
	We construct a path in $\gG_2$ which tracks the image of $[x_0,x_n]$. For each $x_i,x_{i+1}$,
	we take any geodesic in $X(\gG_2)$, $[Q(x_i),Q(x_{i+1})](=[q(x_i),q(x_{i+1})])$, see Figure \ref{f:recongeod}. 
	Because $q$ is a quasi-isometry between
	Cayley graphs, $d_1=\hat{d}_1$ for any segment in $I_1$, 
	and $d_2\geq \hat{d_2}$, we get the following estimate on the lengths of
	the images of endpoints of segments of $I_1$:
\begin{align}
	\sum_{i=0}^{(n-1)/2}\hat{d}_2(q(x_{2i}),q(x_{2i+1}))&\leq
	\sum_{i=0}^{(n-1)/2}d_2(q(x_{2i}),q(x_{2i+1}))\nonumber\\
	\leq\sum_{i=0}^{(n-1)/2}[kd_1(x_{2i},x_{2i+1})+c]
	&=\sum_{i=0}^{(n-1)/2}[k\hat{d}_1(x_{2i},x_{2i+1})+c]\nonumber
\end{align}
	By Lemma \ref{l:horoballqi}, we know that the horoballs paired by $q$ are 
	quasi-isometric and, by \cite{Dr1}, that the constants can be chosen uniformly.
	If we let $\Lambda-2T$ be the maximum among those constants and $k,c$,
	we get the following length estimate:
	
\begin{align}
	\hat{d}_2(q(x_0),q(x_n))&\leq\sum_{i=0}^{(n-1)/2}\hat{d}_2(q(x_{2i}),q(x_{2i+1}))&+&\sum_{i=1}^{(n-1)/2}\hat{d}_2(q(x_{2i-1}),q(x_{2i}))\\
	&\leq\sum_{i=0}^{(n-1)/2}[\Lambda\hat{d}_1(x_{2i},x_{2i+1})+\Lambda]&+&\sum_{i=1}^{(n-1)/2}[\Lambda\hat{d}_1(x_{2i-1},x_{2i})+\Lambda]\nonumber
\end{align}

	Now, since we know that the horoball-transversals have length $\geq1$, we
	can move the additive constant $\Lambda$ from the first sum to the second sum,
	except for a single summand.
	We do this to account for scenarios in which a geodesic is constructed from
	several paths contained entirely in horoballs, making the first sum $\sum\Lambda$.
\begin{align}
  \hat{d}_2(q(x_0),q(x_n))&\leq  \Lambda + \sum_{i=0}^{(n-1)/2}[\Lambda\hat{d}_1(x_{2i},x_{2i+1})]&+&\sum_{i=1}^{(n-1)/2}[\Lambda \hat{d}_1(x_{2i-1},x_{2i})+2\Lambda]\nonumber \\
   &\leq \Lambda + \sum_{i=0}^{(n-1)/2}[\Lambda\hat{d}_1(x_{2i},x_{2i+1})]&+&\sum_{i=1}^{(n-1)/2}[3\Lambda\hat{d}_1(x_{2i-1},x_{2i})] \nonumber \\
   &\leq \Lambda + 3\Lambda\left[\sum_{i=0}^{(n-1)/2}[\hat{d}_1(x_{2i},x_{2i+1})]\right.&+&\left.\sum_{i=1}^{(n-1)/2}[\hat{d}_1(x_{2i-1},x_{2i})]\right] \nonumber \\
   &=3\Lambda\hat{d}_1(x_0,x_n)+\Lambda 
\end{align}
	
	We still need to establish coarse density. However, this is clear for depth 0 vertices
	since the Cayley graphs are quasi-isometric and it is clear for the positive depth
	vertices since the horoballs are in bijection and this pairing is by quasi-isometries.
	
	We also need to produce the other bound, but by a symmetric argument using the
	quasi-inverse $r$,
	$$\hat{d_1}(r(x_0),r(x_n))\leq 3\Lambda'\hat{d}_2(x_0,x_n)+\Lambda'$$
	Therefore, 
	
	$$\hat{d_1}(r(q(x_0)),r(q(x_n)))\leq3\Lambda'\hat{d_2}(q(x_0),q(x_n))+\Lambda' 
	$$ $$\leq 3\Lambda'[3\Lambda\hat{d_1}(x_0,x_n)+\Lambda]+\Lambda' = 
	9\Lambda\Lambda'\hat{d_1}(x_0,x_n)+3\Lambda\Lambda'+\Lambda$$
	
	Because $q$ and $r$ are quasi-inverse, there exists an $a>0$ such that
	
	$$\hat{d_1}(x_0,x_n) \leq \hat{d_1}(r(q(x_0)),r(q(x_n)))+a$$
	Combining these, we get
	
	$$\frac{1}{3\Lambda'}\hat{d_1}(x_0,x_n) - \frac{a+\Lambda'}{3\Lambda'}\leq \hat{d_2}(r(x_0),r(x_n))\leq 3\Lambda'\hat{d_1}(x_0,x_n) + \Lambda'$$
	We conclude by maximizing among constants.
\end{proof}

As indicated previously, our proof of Theorem \ref{t:cuspqi} simplifies to prove an analogous result for the coned space.
\begin{thm}\label{t:coneqi}
	Let $\gG_1,\gG_2,\cA,q$ be as above. Then there exists a collection of subgroups $\cB$
	of $\gG_2$ such that the coned spaces of $(\gG_1,\cA)$ and $(\gG_2,\cB)$ are
	quasi-isometric.
\end{thm}
\begin{proof}
	Adjust the proof of the previous theorem by replacing the intra-horoball arcs with
	arcs through the cone-points. These all have length 2 so simply change (2) to 
	$$\hat{d}_2(q(x_0),q(x_n))\leq\sum_{i=0}^{(n-1)/2}\hat{d}_2(q(x_{2i}),q(x_{2i+1}))+\sum_{i=1}^{(n-1)/2}2$$
\end{proof}

\begin{cor}\label{c:boundaryhomeo}
	With $(\gG_1,\cA_1)$ and $(\gG_2,\cA_2)$ as in Theorem \ref{t:cuspqi}, 
		the cusped spaces $X(\gG_1,\cA_1)$
		and $X(\gG_2,\cA_2)$ have homeomorphic boundaries.
\end{cor}
\begin{cor}\label{c:treesqi}
	With $(\gG_1,\cA_1)$ as in Theorem \ref{t:cuspqi}, 
	the trees describing the maximal peripheral splitting \cite{Bo3} and
	the cut-point/cut-pair tree \cite{PaSw1} for the boundary of the cusped space are
	quasi-isometry invariant.
\end{cor}
\begin{cor}\label{c:gensets}
	Let $\gG$ be a group hyperbolic relative to $\cA$ with finite compatible generating sets 
	$S$ and $T$.
	Then $X(\gG,S,\cA)$ and $X(\gG,T,\cA)$ are quasi-isometric. The
	analogous result for the coned space also holds.
	\end{cor}
\begin{proof} We only need to show that we can drop the condition on peripheral subgroups.
Inspecting the proof of Theorem \ref{t:cuspqi}, the requirement that no peripheral subgroup is properly
relatively hyperbolic is used to establish a bijection between cosets of peripheral subgroups from
\cite{Dr1}. Because peripheral cosets will
be mapped to themselves under the identity map, this bijection is automatic and the
hypothesis is unnecessary. \end{proof}
We note that this corollary is also new for hyperbolic groups with a non-trivial 
relatively hyperbolic structure.

\section{Vertex Stabilizers}
The first step towards identifying the vertex stabilizers is establishing
the quasi-convexity of vertex groups.

\begin{lem}\label{l:qcvg}
	Let $\gG$ be finitely presented and one-ended. Additionally suppose that
	$(\gG,\cA)$ is relatively hyperbolic with $\cA$ finite such that no $A\in\cA$ is properly relatively
	hyperbolic and no $A$ contains an infinite torsion subgroup. Let $\cT$ be the combined tree from
	the boundary. If $\gG_v$ is a vertex group of $\cT$ then $\gG_v$ is relatively quasi-convex.
\end{lem}

\begin{proof}
	This is clearly true for vertex groups which are peripheral. It is also true for hyperbolic two-ended
	vertex groups by
	Theorem \ref{t:2endsrq}. 
	Assume $\gG_v$ is not of these types.
	
	$\cT$ is a bipartite graph in which all vertices of one color have corresponding groups which
	are either peripheral or hyperbolic two-ended. To see this, note that maximal inseparable sets and necklaces must only be adjacent to cut pairs or points. Furthermore, cut pairs and points must come between these larger sets. Now, let $\{Z_1,\ldots,Z_n\}$ and 
	$\{P_1,\ldots, P_m\}$ be the
	collection of all hyperbolic two-ended subgroups and the 
	collection of all peripheral subgroups incident to $\gG_v$, respectively.

	Let $\gga$ be any geodesic between points in $\gG_v$. Decompose $\gga$ into maximal segments
	of length $\geq 1$ which have endpoints in cosets of some $Z_i$,
	$P_j$ or are contained completely within $\gG_v$. The segments contained in each $Z_i$ must stay within a
	bounded distance of $\gG_v$ because hyperbolic two-ended subgroups are
	strongly relatively quasi-convex (Theorem \ref{t:2endsrq}). Additionally the $P_j$ segments
	stay within bounded distance of $P_j$ \cite[Theorem 4.21]{Dr1}. Since they are peripheral and we are investigating
	relative quasi-convexity, we are not concerned
	with how far they stray from $P_j\cap \gG_v$ inside $P_j$.
\end{proof}

\begin{prop}\label{p:necklaceQH}
	With $\gG,\cA,\cT$ as in Lemma \ref{l:qcvg}, 
	a vertex group $\gG_v$ of 
	$\cT$, is relatively QH
	with finite fiber if and only if $\gG_v$ is the stabilizer of a necklace in $T$.
\end{prop}

\begin{proof}
	If $\gG_v$ is relatively QH with finite fiber, then by Definition \ref{d:QH} there is a short exact sequence
	$$1\to F \to \gG_v \to \pi_1(\cO)\to 1$$
	with $F$ finite and $\cO$ a hyperbolic orbifold.
	We first determine that $\gG_v\not\in\cA$. Because $F$ is finite,
	$\gG_v\simeq_{qi}\pi_1(\cO)$ and because $\cO$ is a hyperbolic
	2-orbifold, $\gG_v$ is itself hyperbolic relative to $\{1\}$.
	By our condition on peripheral subgroups, $\gG_v\not\in\cA$.
	
	Now, by definition $\gG_v$ is virtually Fuchsian.
	Let $\cC$ be the set of bi-infinite curves in the universal cover
	$\widetilde{\cO}$ which are not homotopic to a
	boundary component of $\widetilde{\cO}$. Since $F$ is finite, $\gG_v$ is quasi-isometric
	to $\pi_1(\cO)$ and so $\partial \gG_v\simeq\partial\pi_1(\cO)$ with respect to the relatively hyperbolic structure induced by $\cA$. Call this set $N$. Let $\cN$ be
	the image of $N$ in $\partial\gG$ induced by the inclusion of $\gG_v$ in $\gG$. This
	map is well-defined because of relative quasi-convexity of $\gG_v$. Specifically,
	by \cite{MaPe1}, relative quasi-convexity is equivalent to quasi-convexity in 
	the cusped space and by general results on $\gd$-hyperbolic spaces the boundary
	will embed.
	
	We claim that $\cN$ is a necklace in $\partial\gG$. By definition, 
	every edge group must be either
	finite or contained in a boundary component. Because $\gG$ is one-ended, the
	finite case is excluded. Consequently, for any $\gga\in\cC$ each coset of an edge 
	group is contained in a single component of $\widetilde\cO\setminus \gga$. Let $\eta\in\cC$ be any
	curve which has an essential crossing with $\gga$. Such $\eta$ exists because if not then $\gga$ must be boundary parallel, but $\cC$ contains
	no boundary parallel curves. 
	
	Since $\eta^+$
	and $\eta^-$ are in different components of $N\setminus\{\gga^+,\gga^-\}$ and
	each edge group is attached to only a single boundary component, it must be
	that every edge group has image contained in either the same component of 
	$N\setminus\gga$ as $\eta^+$ or $\eta^-$, but never both. Thus, the image of
	$\{\gga^+,\gga^-\}$ also separates the image of $\{\eta^+,\eta^-\}$ in $\cN$
	and the endpoints of $\gga$ form a cut pair in $\partial\gG$.
	
	\begin{figure}[t]\label{f:QH} \includegraphics{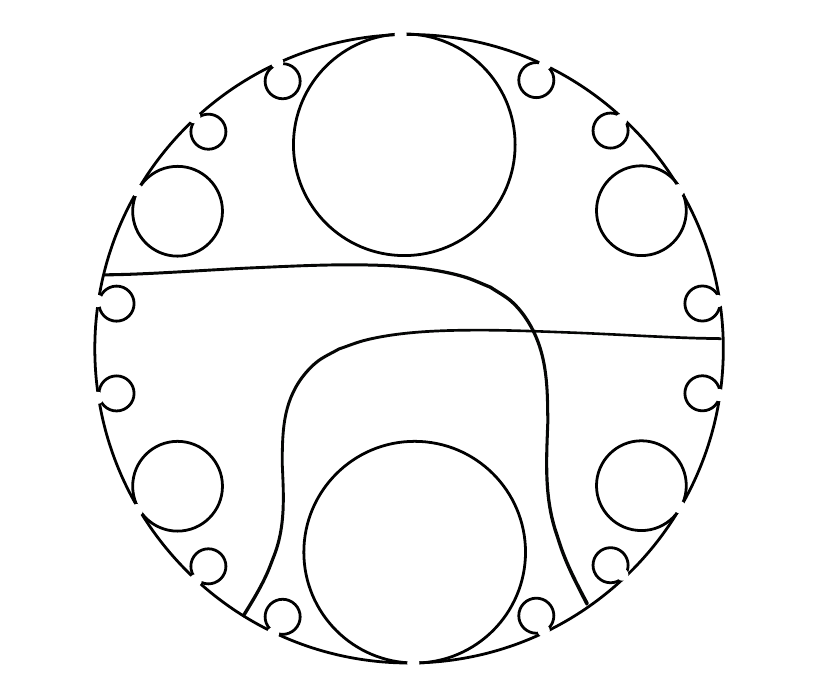}
	\caption{Interlocking geodesics separate boundary components.}	
	\end{figure}
	In the reverse direction, if $\gG_v$ stabilizes a necklace $\cN$ then, 
	by the last paragraph of the proof of \cite[Theorem 22]{PaSw1}, 
	it has an action on $S^1$ which preserves
	the cyclic order. Since $\gG_v=\mathrm{Stab}(\cN)$, $\gG_v$ inherits the convergence
	property of $\gG$ and this property must be realized on $\cN$. 
	Any sequence of group elements contained in the kernel of this
	action is not a convergence sequence, so the kernel must be finite. 
	The fiber, $F$, is the kernel of this action. By Theorem \ref{t:convergence},
	the quotient must be a Fuchsian group. Let $\cO$ be the quotient
	of $\bbH^2$ by the action of $\gG_v/F$,
	truncating cusps so that $\cO$ is compact. We are left with showing
	that edge groups are boundary parallel.
	
	First, we show that no element of $\cC$ (again defined as the set of non-boundary 
	bi-infinite curves on $\cO$ - those which cross other bi-infinite curves and interlock in 
	$\partial N$) can be contained in the image
	of a peripheral edge group. Because
	peripheral subgroups have a unique boundary point (Lemma \ref{l:horoballboundary}), such 
	a curve would induce a cut point in $\cN$.
	However, by Lemma \ref{l:cpsepcp} no cut pair can be separated by a cut point.
	In this context, every cut pair separates another cut pair with the only
	exception arising from those cut pairs which are end points 
	of boundary curves of $\widetilde\cO$. In other words, for every cut pair $C$ of $\cN$
	there is an interlocked cut pair unless $C$ forms the endpoints in
	$\widetilde{\cO}$ of a curve homotopic to a boundary component of $\cO$.
	
	Now we are left with only the possibility that some $\gga\in\cC$ is
	identified with a hyperbolic two-ended edge group. 
	Let $\{x_1,x_2\}$ be any cut pair interlocked with $\{\gga^+,\gga^-\}$.
	We claim that $\{x_1,x_2\}$ is not actually a cut pair in $\partial\gG$.
	$\partial\gG\setminus\{\gga^+,\gga^-\}$ must have at least 3 components
	in this situation. Let $Y$ be a component which does not contain any
	points of $\cN$. In particular, $Y\cap\cN=\{\gga^+,\gga^-\}$.
	Thus, no cut pair of $\cN$ separates $\gga^+$ from $\gga^-$ as
	both are contained in the component $Y$. However, then $\overline{Y}$
	must separate $\{x_1,x_2\}$, contrary to Lemme \ref{l:cpsepcp}.
	Thus, we must have that no such $\gga$ exists in $\cC$.
	
	\end{proof}

\begin{lem}\label{l:cp2end}
	If $\{x,y\}$ is an inseparable cut pair in $\partial\gG$ then $\mathrm{Stab}(\{x,y\})$ 
	is a hyperbolic two-ended subgroup of $\gG$.
\end{lem}

\begin{proof}
	Let $Z=\mathrm{Stab}(\{x,y\})$. $Z$ must be infinite else $\gG$ would not be
	one-ended. As a subgroup of $\gG$, 
	$Z$ acts on $\partial\gG$ with the convergence property.
	Since $Z$ fixes $\{x,y\}$, this implies that $Z$ is two-ended.
\end{proof}

\begin{cor}\label{c:correspondence}
	With $\gG,\cA,\cT$ as in Lemma \ref{l:qcvg}, there is a correspondence between vertex groups of $\cT$
	tree and the topological features of the boundary given by Table 2.
\begin{table}
\caption{Subgroups which stabilize particular topological features in the boundary.}

	\centering
	\begin{tabular}{r c l} \hline\hline Stabilizer & &Topological Feature\\ \hline
	hyperbolic 2-ended&$\longleftrightarrow$& cut-pair \\
	peripheral&$\longleftrightarrow$& cut-point \\
	relatively QH with finite fiber&$\longleftrightarrow$& necklace \\
	\end{tabular}	
\end{table}

\end{cor}
\begin{proof}
	Since hyperbolic 2-ended subgroups are strongly relatively quasi-convex \cite{Os1},
	their boundaries embed \cite{Hr1}. Similarly, \cite{Bo5} demonstrates that cut points
	correspond to peripheral splittings. As vertex groups, these features must be separating.
	The last point is Proposition \ref{p:necklaceQH}.
	\end{proof}

\noindent With this in place, we are ready to show:

\begin{thm}\label{t:lastthm}
		Let $\gG_1$ and $\gG_2$ be finitely generated groups. Suppose additionally 
		that $\gG_1$ is one-ended and hyperbolic relative to the finite collection $\cA_1$ 
		of subgroups such that no $A\in\cA$ is
		properly relatively hyperbolic or contains an infinite torsion subgroup. 
		Let $\cT$ be the cut-point/cut-pair tree
		of $\partial(\gG_1,\cA_1)$. If $f:\gG_1\to\gG_2$ is a quasi-isometry then
		\begin{itemize}
			\item $T$ is the cut-point/cut-pair tree for $\gG_2$ with respect to the peripheral structure
			induced by Theorem \ref{t:cuspqi},
			\item if $\text{Stab}_{\gG_1}(v)$ is one of the following types then $\text{Stab}_{\gG_2}(v)$ is of the same type,
			\begin{enumerate}
				\item hyperbolic 2-ended,
				\item peripheral,
				\item relatively QH with finite fiber.
			\end{enumerate}
		\end{itemize}
	\end{thm}

\begin{proof}
	By Corollary \ref{c:boundaryhomeo}, there exists a relatively hyperbolic structure for $\gG_2$ such that
	the boundaries of the cusped spaces are homeomorphic. Since $T$ depends only
	on the topology of this continuum, $\cT$ is the cut-point/cut-pair tree for $\gG_2$.
	
	By the correspondence given in Corollary \ref{c:correspondence}, these vertex types depend only
	on the topology of the boundary. Since these topological features are preserved,
	the vertex group types are preserved as well. 
	
\end{proof}

We conclude this section with a consequence of the fact that $\mathrm{Out}(\gG)$ acts on
$\partial\gG$ by homeomorphisms.

\begin{cor}\label{c:outinv}
	Let $(\gG,\cA)$ be relatively hyperbolic with no $A\in\cA$ properly relatively
	hyperbolic. Then, 
	the $\mathrm{Out}(\gG)$ action on the JSJ-deformation space over elementary
	subgroups relative to peripheral subgroups fixes $T$.
\end{cor}

\begin{proof}
	The action passes to an action on the boundary by homeomorphisms which 
	induces an action on the combined tree so that
	vertex groups map to vertex groups and adjacencies are preserved. Furthermore,
	because maximal relatively hyperbolic structures are unique \cite{MaOgYa1},
	there is no change in the choice of peripheral structure.
\end{proof}

As demonstrated in Theorem \ref{t:cuspqi}, the action of $\cQ\cI(\gG)$ on the Cayley graph of $\gG$
induces an action on the cusped space $X(\gG,\cA)$ and, because the boundary is invariant under quasi-isometries, on $\partial(\gG)$. As stated in Corollary \ref{c:treesqi}, this naturally 
yields an action on the tree $\cT$.
\section{Splitting $\cQ\cI(\gG)$ with $\cT$}
Our last section describes how to understand $\cQ\cI(\gG)$ from the action it inherits on
$\cT(\gG)$.
\subsection{Faithfulness}
\begin{thm}\label{QIfaith}
	With $\gG,\cA,\cT$ as in Lemma \ref{l:qcvg}, the action of $\cQ\cI(\gG)$ on $\cT(\gG,\cA)$ is faithful, assuming that $\cT$ is not
	a point.
\end{thm}
\begin{proof}

	Given a $(k,c)$-quasi-isometry $\phi\in\cQ\cI(\gG)$ which has a trivial induced action on
	$\cT(\gG)$, it must be that $\phi$ coarsely fixes all peripheral subgroups over
	which $G$ splits. For now, assume one such subgroup exists and call this $A$. By the same citation 
	of \cite{Dr1} applied to Theorem
	\ref{t:cuspqi}, $\phi$ must map every coset $gA$ within $N_T(gA)$.
	However, given two adjacent cosets, $hA$ and $shA$, both must
	satisfy this condition. Consider adjacent points $a\in hA$, $b\in shA$. It must be that
	$\phi(a)\in N_T(hA)$ and $\phi(b)\in N_T(shA)$. Further, as 
	$d(a,b)=1$, $d(\phi(a),\phi(b))\leq k+c$.
	Consequently, $\phi(a)\in N_{k+c+T}(shA)$ and $\phi(b)\in N_{k+c+T}(hA)$ and
	$\phi(a),\phi(b)\in N_{k+c+T}(hA)\cap N_{k+c+T}(shA)$. This intersection
	must have a finite diameter which is uniformly bounded over any chosen pair
	of cosets. Therefore, $d(a,\phi(a))$ is bounded by this diameter.
	Since every point can be found within such a neighborhood by an appropriate
	choice of such cosets,
	$\phi$ must be a finite distance from the identity map.
	
	The two-ended case is similar. These subgroups have been shown to satisfy
	a similar divergence property called \emph{hyperbolically embedded}
	\cite{DaGuOs1,Si2} and they are strongly relatively quasi-convex, so
	the argument is nearly identical.
	
	\end{proof}

\begin{figure}[h]
\includegraphics[width=5.6in]{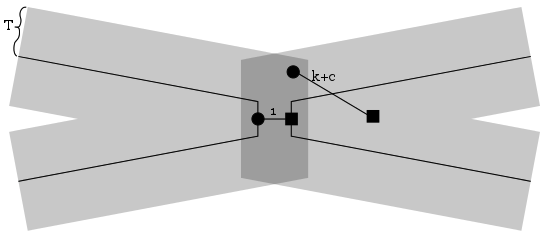}
\caption{$\cQ\cI(\gG)$ acts faithfully on the JSJ tree.}
\end{figure}
\subsection{Bounding the Number of Edges of $\cT/\cQ\cI(\gG)$}
We have that $\gG$ and $\cQ\cI(\gG)$ act on the same tree so that $\cQ\cI$ splits
whenever $\gG$ has a non-trivial JSJ as described here. We can also control the
number of edges which $\cQ\cI$ admits in this induced splitting.
Let $\Lambda$ be $|E(\cT/\gG)|$ and let $\mathrm{Aut}_{qi}(\cT/\gG)$ be the group
of graph automorphisms which respect the quasi-isometry type of each edge and vertex stabilizer.

\begin{thm}\label{t:edges}
The graph of groups decomposition of $\cQ\cI(\gG)$ induced by the JSJ-decomposition of
$\gG$ has at most $\Lambda$ edges and at least $\Lambda / |\mathrm{Aut}_{qi}(\cT/\gG)|$ edges.
\end{thm}
\begin{proof}
First, note that $\gG$ acts
	on itself by quasi-isometries and if $\cT$ is not a point then this
	action will have no global fixed points so obviously the same is true for the
	$\cQ\cI(\gG)$ action.
	
Now, each $g\in\gG$ acts by (quasi)-isometries on
$X(\gG,\cA)$ which gives the upper bound, and $\cQ\cI(\gG)$ acts on $\cT/\gG$ by graph automorphisms which
must preserve the quasi-isometry class of each stabilizer because the action is filtered through the action on $X(\gG,\cA)$, thus giving the lower bound.
\end{proof}

\section{Acknowledgements}
I would like to thank Daniel Groves for his generous support and guidance. I would also like to thank Mark Feighn for pointing out an unstated hypothesis for the peripheral subgroups in Sections 3 and 5, Panos Papazoglou and Eric Swenson for some helpful remarks concerning \cite[Section 5]{PaSw1} and Chris Cashen for his useful comments on my first version.
\bibliography{bib}{}
\bibliographystyle{alpha}

\end{document}